\documentclass[12pt]{amsart}

\usepackage{amssymb}
\usepackage{stmaryrd}
\usepackage{fullpage}	
\usepackage{amscd}   	
\usepackage[all]{xy} 		
\usepackage{comment} 	
\usepackage{mathrsfs}      

\DeclareFontEncoding{OT2}{}{} 

\usepackage{tikz}
\usetikzlibrary{decorations.pathmorphing}

\usepackage{color}

 \newcommand{\easterEgg}[1]{}

\def\ce{\ensuremath{\mathcal{E}}}
\def\cd{\ensuremath{\mathcal{D}}}
\DeclareMathOperator{\sm}{{sm}}
\DeclareMathOperator{\chab}{{chab}}

\DeclareMathOperator{\ur}{ur}
\DeclareMathOperator{\cj}{\mathcal{J}}
\DeclareMathOperator{\num}{num}
\DeclareMathOperator{\toric}{tor}

\DeclareMathOperator{\Cliff}{Cliff}

\newcommand{\defi}[1]{\textsf{#1}} 				

\newcommand{\F}{{\mathbb F}}
\newcommand{\G}{{\mathbb G}}

\newcommand{\Q}{{\mathbb Q}}

\newcommand{\Z}{{\mathbb Z}}

\newcommand{\mdeg}{{\underline{\deg}}}

\newcommand{\pp}{{\mathfrak p}}

\newcommand{\calD}{{\mathcal D}}
\newcommand{\calE}{{\mathcal E}}

\newcommand{\calL}{{\mathcal L}}
\newcommand{\calM}{{\mathcal M}}

\newcommand{\calO}{{\mathcal O}}

\newcommand{\OO}{{\mathcal O}}

\newcommand{\scrL}{{\mathscr L}}

\newcommand{\scrX}{{\mathscr X}}

\def\Q{\mathbb{Q}}

\def\P{\mathbb{P}}

\def\Z{\mathbb{Z}}

 \DeclareMathOperator{\divv}{div}
\DeclareMathOperator{\ord}{ord} \DeclareMathOperator{\Sym}{Sym}
\DeclareMathOperator{\Div}{Div} \DeclareMathOperator{\Pic}{Pic}
 
\DeclareMathOperator{\Spec}{Spec}  
 \DeclareMathOperator{\rank}{rank}

\newcommand{\Ab}{\operatorname{ab}}

\newcommand{\ab}{{\operatorname{ab}}}

\newcommand{\isom}{\simeq}

\numberwithin{equation}{section}
\newtheorem{theorem}[equation]{Theorem}
\newtheorem{lemma}[equation]{Lemma}
\newtheorem{corollary}[equation]{Corollary}
\newtheorem{proposition}[equation]{Proposition}

\theoremstyle{definition}
\newtheorem{definition}[equation]{Definition}

\newtheorem{example}[equation]{Example}

\theoremstyle{remark}
\newtheorem{remark}[equation]{Remark}

\usepackage[
	backref,
	pdfauthor={David Brown}, 
]{hyperref}

\usepackage[alphabetic,backrefs,lite]{amsrefs} 

 \newcommand{\marg}[1]{}

\newcommand{\Supp}{\operatorname{Supp}}
\newcommand{\OK}{\calO_{K_{\pp}}}
\newcommand{\omX}{\omega_{\pi}}
\newcommand{\homX}{H^0(\scrX,\omX)}

\begin{document}

\title[The Chabauty-Coleman bound at a prime of bad reduction and Clifford bounds for Geometric Rank Functions]
{The Chabauty-Coleman bound at a prime of bad reduction and Clifford bounds for Geometric Rank Functions}
\author{Eric Katz}
\address{Department of Combinatorics \& Optimization, University of Waterloo, 200 University Avenue West, Waterloo, ON, Canada N2L 3G1} \email{eekatz@math.uwaterloo.ca}
\author{David Zureick-Brown}
\address{Dept. of Mathematics and Computer science, Emory University,
Atlanta, GA 30322 USA}
\thanks{The first author was partially supported by a National Sciences and Engineering Research Council of Canada grant.  The second author was partially supported by a National Defense Science and
Engineering Graduate Fellowship and by a National Security Agency Young Investigator grant.}
\urladdr{http://mathcs.emory.edu/\~{}dzb/}

\date{\today}

\begin{abstract}
    Let $X$ be a curve over a number field $K$ with genus $g \geq 2$, $\pp$ a prime of $\calO_K$ over an unramified rational prime $p > 2r$, $J$ the Jacobian of $X$, $r = \rank J(K)$, and $\scrX$ a regular proper model of $X$ at $\pp$. Suppose $r < g$. We prove that $\#X(K) \leq \#\scrX(\F_{\pp}) + 2r$, extending the refined version of the Chabauty-Coleman bound to the case of bad reduction. The new technical insight is to isolate variants of the classical rank of a divisor on a curve which are better suited for singular curves and which satisfy Clifford's theorem.
\\

\end{abstract}

\maketitle


\newcommand{\Xp}{X_{\pp}}

\section{Introduction}
\label{S:introduction}    Let $K$ be a number field and $X/K$ be a
    curve (i.e. a smooth geometrically integral 1-dimensional variety) of genus
    $g \geq 2$ and let $p$ denote a prime which is unramified in $K$. 
    Faltings' \cite{faltings}, Vojta's \cite{Vojta}, and Bombieri's \cite{Bomb} 
    proofs of the Mordell Conjecture tell us that $X(K)$ is finite, but are ineffective, providing no assistance 
    in determining $X(K)$ explicitly for a specific curve. Chabauty \cite{chab}, 
    building on an idea of Skolem \cite{sko34}, gave a partial proof of the Mordell Conjecture
    under the hypothesis that the rank $r$ of the Jacobian of $X$ is strictly less than the   
    genus $g$. Coleman later realized that Chabauty's proof could be
    modified to get an explicit upper bound for $\#X(K)$.
    \begin{theorem}
      [\cite{Cole_eff}]
      \label{T:Coleman}
       Suppose $p > 2g$ and let $\pp \subset \calO_K$ be a prime of good
      reduction which lies above $p$. Suppose $r < g$. Then
      $$\#X(K) \leq \#X(\F_{\pp}) + 2g - 2.$$
    \end{theorem}
    In \cite{LT_thue}, the authors ask if one
    can refine Coleman's bound when the rank is small (i.e. $r \leq
    g-2$). Stoll proved that by choosing, for each residue class, the `best' differential
    one can indeed refine the bound.
    \begin{theorem} [\cite{stoll_indep}*{Corollary 6.7}]       
     \label{T:stollthm}
      With the hypothesis of Theorem \ref{T:Coleman},
      \[
        \#X(K) \leq \#X(\F_{\pp}) + 2r.
      \]
    \end{theorem}
    In
    another direction,  Lorenzini and Tucker derive Coleman's bound
    when $\pp$ is a prime of bad reduction;
       a later,  alternative proof using   
    intersection theory on $\scrX$ is given in \cite{PMc_survey}*{Theorem A.5}.
Let $\scrX$ be a minimal regular proper model  of $X$ at $\pp$ and denote by 
    $\scrX_{\pp}^{\text{sm}}$ the smooth locus of $\scrX_{\pp}$. 

    \begin{theorem}[\cite{LT_thue}*{Proposition 1.10}] 
    \label{T:PMC}
      Suppose $p > 2g$ and let $\pp$ be
     a prime above $p$. Let $\scrX$ be a proper regular model of $X$ over 
     $\calO_{K_{\pp}}$. Suppose $r < g$. Then
      $$\#X(K) \leq \#\scrX_{\pp}^{\sm}(\F_{\pp}) + 2g - 2.$$
    \end{theorem}
    
 \subsection{Main Result} 


Lorenzini and Tucker ask \cite{LT_thue}*{comment after 1.11} if one can generalize Theorem
    \ref{T:stollthm} to the case when $\pp$ is a prime of bad
    reduction; Michael Stoll remarks \cite{stoll_indep}*{Remark 6.5}
    that the analogue of Theorem \ref{T:stollthm} is true at least for a hyperelliptic
    curve. The
    main result of this paper is such a generalization.
   \begin{theorem}
     \label{T:mainthm}
    Let $X$ be a curve over a number field $K$. Suppose $p > 2r+2$ is a prime which is unramified 
    in $K$  and let $\pp \subset \calO_K $ be a prime above $p$. Let $\scrX$ be a proper regular
     model of $X$ over $\calO_{K_{\pp}}$.
Suppose moreover that $r < g$. Then
     \[
	     \#X(K) \leq \#\scrX_{\pp}^{\sm}(\F_{\pp}) + 2r.     
     \]
    \end{theorem}


    \begin{remark}
The condition that $p > 2r+2$ ensures that the bounds one gets from the use of Newton Polygons in Chabauty's method are as sharp as possible; 
using Proposition \ref{P:residueBound}, one can write out weaker, but still explicit (in terms of $g$
    and $p$), bounds when $p \leq 2r+2$ or when $p$ ramifies in $K$ (see any of \cite{Cole_eff},
    \cite{stoll_indep}, or \cite{LT_thue}). 
    \end{remark}

    \begin{remark}
      When $\scrX_{\pp}$ is singular or is not hyperelliptic, one can often do
      better than $2r$; see Subsection \ref{S:remarks}.
    \end{remark}


\begin{remark}
    The charm of these theorems is that they occasionally allow one to compute $X(K)$; see \cite{grant} for the first such example and Example \ref{ex:sharp} for another which uses the refined bound of Theorem \ref{T:mainthm} at a prime of bad reduction.

Example \ref{ex:sharp} is of course quite special, in that most of the time the bound of Theorem \ref{T:mainthm} is not sharp. Chabauty's method gives a way to bound the number of rational points in each residue class (see Proposition \ref{P:residueBound}); however, this bound is never less than 1, and yet some residue classes may not have rational points in them. In practice, one applies Chabauty's method to each residue class. If a given residue class seems to have no rational points, a variety of methods usually allow one to prove that this is correct (the most common being Scharaschkin's method, often called the Mordell-Weil sieve); if a given residue class does have a point,  then one applies Chabauty's method.

There are a wealth of interesting examples where a more careful analysis of Chabauty's method allows one to determine $X(K)$. Worth noting are the works \cite{bruin_chab} and \cite{pss}, where the integral coprime solutions of the generalized Fermat equations $x^2 + y^8 =z^3$ and $x^2 + y^3 = z^7$ (both of which have large, non-trivial solutions) are completely determined by reducing to curves and using Chabauty methods. 
\vspace{3pt}

    For low genus hyperelliptic curves Chabauty's method has been made completely explicit and even implemented in \cite{Magma}, so that for any particular low genus hyperelliptic curve $X$ over $\Q$, $X(\Q)$ can often be determined; see the survey \cite{poonen_survey} for a general discussion of computational issues, \cite{PMc_survey}*{Section 7} for a discussion of effectivity of Chabauty's method, and the online documentation \cite{Magma} for many details. 
See \cite{stoll:rationalPointsSurvey} for the most up to date survey on such issues -- \cite{stoll:rationalPointsSurvey}*{4.2} gives a quick summary of the interplay between the Mordell-Weil Sieve and Chabauty's method and how to handle residue classes which are empty, and  \cite{bruinS:MWSieve} and \cite{wee} give a much more  detailed explanation.


\end{remark}

\begin{remark}
	The $p=2r+2$ case of Theorem \ref{T:mainthm} (i.e. $p = 2$ and $r = 0$) is \cite{LT_thue}*{Proposition 1.10}.
\end{remark}
        
 \subsection{New ideas -- geometric rank functions} 

Recall the classical notion of the rank of a divisor $D$ on a smooth projective
geometrically integral curve $X$ over an infinite field.
\begin{definition}
\label{d:classicalRank}

We say that the \defi{rank} $r_X(D)$ of a divisor $D \in \Div X$ is $-1$ if $D$ is not equivalent to an effective divisor, and we say that a divisor $D$ which is equivalent to an effective divisor has rank $r_X(D) = r$ 
if for any effective divisor $E$ of degree $r$ on $X$, $D-E$ is equivalent to an effective divisor and moreover $r$ is the largest integer with this property (i.e. there is an effective divisor $E$ on $X$ of degree $r+1$ for which $D-E$ is not equivalent to an effective divisor).
Alternatively, let 
\[
r(D) = \dim_K H^0(X,\calO_{X}(D))-1 =  \dim|D|, \text{ where } 	|D|
= \P(H^0(X,\calO_{X}(D)))
\]
($|D|$ is the linear system of effective divisors equivalent to $D$).
\end{definition}
Since $X$ is smooth, projective,  geometrically integral, and defined over an infinite field, the functions $r_X$ and $r$ are equal. Moreover, $r_X(D)$ satisfies \emph{Clifford's Theorem}
\cite{hart}*{Chapter III, Theorem 5.4}: if $D$
is a \emph{special} divisor (i.e. $D$ and $K_X-D$ are each equivalent to an
effective divisor, where $K_X$ is the canonical divisor of $X$), then
$r_X(D) \leq \frac{1}{2}\deg D$. 
\vspace{4pt}

Now let $X$ be the curve from the setup of Theorem \ref{T:stollthm}. Stoll's proof of Theorem
\ref{T:stollthm} uses Clifford's theorem and Riemann-Roch in an
essential way -- he constructs a special divisor $D_{\chab,\pp}$ on the (smooth) special fiber $\scrX_\pp$ of $X$ (see Subsection \ref{ss:chabautyDivisor}) such that
\begin{itemize}
\item [(i)]   $\#X(K) - \#X(\F_{\pp}) \leq \deg D_{\chab,\pp}$, and
\item [(ii)] $r(K_X-D_{\chab,\pp}) \geq g-r-1$
\end{itemize}
and deduces that $\deg D_{\chab,\pp} \leq 2r$ from Riemann-Roch and Clifford's theorem.
 \vspace{3pt}

 When the special fiber $\scrX_\pp$ has multiple components, it
 is no longer true that the ranks $r_X(D)$ and $r(D)$ agree for divisors $D$ on $\scrX_{\pp}$, since a section $s \in H^0(\scrX_{\pp},\calO_{\scrX_{\pp}}(D))$ may
 vanish along some component of $\scrX_{\pp}$; in general $r_X(D) < \dim_K  H^0(\scrX_{\pp},\calO_{\scrX_{\pp}}(D))$. 
When $\scrX$ is regular, $r(D)$ still satisfies Riemann-Roch \cite{liu}*{Theorem 7.3.26}, and in an
earlier version of this paper we gave a direct proof of Clifford's
theorem for $r(D)$. However, the divisor $D_{\chab,\pp}$ to which we want to apply Clifford's theorem
is no longer special in the usual sense; moreover, when $\scrX_{\pp}$
has multiple components, it does \emph{not} follow from Chabauty's
method that 
$r(K_{\scrX_{\pp}}-D_{\chab,\pp}) \geq g-r-1$ 
(compare  with Subsection \ref{ss:upperBounds}).
\\

 \textbf{Modified rank functions}: In Subsection \ref{ss:rankFunctions}, motivated by the canonical
 decomposition of the special fiber of the N\'{e}ron model of the
 Jacobian of  $X$ 
as successive extensions of Abelian, toric, unipotent, and discrete algebraic groups
(see Subsection \ref{ss:motivation}), we define rank functions
  $r_{\num},r_{\Ab},r_{\toric}, r_X$ which satisfy  $r_{\num}(D) \geq r_{\Ab}(D) \geq r_{\toric}(D) \geq  r_X(D)$. (Our variant of $r_X$ agrees with the one defined above when $X$ has good reduction; moreover $r_X = r$ (Proposition \ref{L:upperBounds}).)



To take into account that sections can vanish along components of
$\scrX_{\pp}$, these functions instead take as input divisors on
$\scrX$ which reduce to smooth points of  $\scrX_\pp$.  Moreover, the
modified rank functions should be viewed as successively finer \emph{approximations} to the effectivity of a divisor. 
For example, if a horizontal divisor $\calD$ on $\scrX$ is equivalent to an
effective horizontal divisor $\calE$, then the degree $\deg
(\calE|_{X_v})$ of the
restriction of $\calE$ to each component $X_v$ of $\scrX_{\pp}$ is
non-negative. So a \emph{necessary} condition  for a horizontal divisor $\calD$ on $\scrX$ to be equivalent to an
effective horizontal divisor $\calE$ is that $\deg (\calE|_{X_v}) \geq
0$ for every component $X_v$ of $\scrX_{\pp}$; to define $r_{\num}(\calD)$ we thus modify Definition \ref{d:classicalRank} to say that $r_{\num}(\calD)$ is at least $r$ if for every horizontal divisor $\calE$ of degree $r$, $\calD-\calE$ has non-negative degree on \emph{every} component $X_v$ of $\scrX_{\pp}$, and that $r_{\num}(\calD) = -1$ if $\deg (\calD|_{X_v}) <
0$ for \emph{some} component $X_v$ of $\scrX_{\pp}$ . (Actually, we
also allow \emph{twisting} -- i.e. modification by divisors $\sum
a_vX_v$ supported on the special fiber; see Subsection \ref{ss:dualGraphs}.)
\\



Our main result about these rank functions
(Proposition \ref{p:clifford}) is that $r_{\ab}$ (and thus
$r_{\toric}(D)$ and $r_X(D)$) satisfies a variant of Clifford's theorem:
\begin{proposition} 
Let $D_\pp$ be an effective divisor on $\scrX_\pp$ and let $K_{\scrX_\pp}$ be a divisor in the class of the canonical bundle.  Suppose that $D_\pp$ and $K_{\scrX_\pp}$ are supported on smooth points.  Then,
\[r_{\Ab}(K_{\scrX_\pp}-D_\pp)\leq g-\frac{\deg{D_\pp}}{2}-1.\]
\end{proposition}

With a bit more work, Theorem \ref{T:mainthm} follows (see Section \ref{S:finalTouch}).

\begin{remark}[Comparison with Baker-Norine]

Baker \cite{BakerSpec} defines a notion of linear equivalence, and
thus of rank, for a divisor on a graph. Moreover, one can specialize a
divisor $D$ on a semistable curve to a divisor $\mdeg(D)$  on its dual graph; see Subsections \ref{ss:dualGraphs} and \ref{S:bakerNorineBackground}.

Our rank function $r_{\num}$ agrees with theirs in the sense that $r_{\num}(D) = r(\mdeg(D))$.
 When $\scrX_{\pp}$ is semistable and totally degenerate, $r_{\num}(D) =
  r_{\ab}(D)$, but in general $r_{\num}(D)$ is larger (see Example \ref{E:numAb}). In \cite{BakerN:RR}, Baker
  and Norine prove  an analogue of Riemann-Roch (and thus Clifford's
  theorem) for $r_{\num}(D)$, but this is with respect to the
  canonical divisor $K_{\Gamma}$ of the dual graph. This
  differs from the specialization of the divisor $K_{\scrX_\pp} $ and
  in particular does not take into account the genera of the
  components of $\scrX_{\pp}$. Their theorem thus suffices for
  $\scrX_{\pp}$ semistable and totally degenerate, but not for the general case.


\end{remark}

\begin{remark}
While this preprint was in preparation, we became aware of the work of Amini-Baker on metrized complexes of curves \cite{AB_metrized}.  They introduce a rank function analogous to $r_{\Ab}$ in a very  slightly different context and prove a version of the Riemann-Roch theorem that implies both the Riemann-Roch theorem for $r_{\Ab}$ and our Clifford bounds on $r_{\Ab}$.  Their work also makes connections to the theory of limit linear series.
\end{remark}

 \subsection{Organization of the paper} 

     This paper is structured as follows. In Section \ref{S:method} we review the method 
    of Chabauty and Coleman and present the main argument used to bound $\#X(K)$. 
In Section \ref{S:geometricRank} we introduce the geometric rank functions $r_{\num},r_{\Ab},r_{\toric}, r_X$ and prove Clifford's theorem for $r_{\ab}$ (Proposition \ref{p:clifford}).
In Section \ref{S:finalTouch} we finish the proof of Theorem \ref{T:mainthm}.
    Finally, in Section \ref{S:examples} we give 
   an example where the refined bound of Theorem \ref{T:mainthm} can be used to determine   $X(K)$ and examples where the inequalities $r_{\num}(D) \geq r_{\Ab}(D) \geq r_{\toric}(D) \geq  r_X(D)$ are strict.

 \section{The Method of Chabauty and Coleman}
 \label{S:method}
	In this section we recall the method of Chabauty and Coleman.
	See \cite{PMc_survey} for many references and a more detailed 
	account.\\

	\textbf{Some notation.} Let $K$ be a number field, 
	 $p$ a prime integer and $\pp$ a prime of $K$ above $p$; let
     $v = v_{\pp}$ be the corresponding valuation,  normalized so that the value
     group $v(K)$ is $\Z$. 
 Denote by $\F_\pp$ the residue field. For a scheme $Y$ over $K$ let 
	$Y_{\pp}$ be the extension of scalars $Y \times_K K_{\pp}$, where $K_{\pp}$ is the completion 
	of $K$ at $\pp$. For a scheme $Y$ over a field
	denote by $Y^{\text{sm}}$ its smooth locus. Let $X$ be a smooth projective geometrically 
	integral curve of genus $g \geq 2$ over $K$ with Jacobian $J$ and let $r = \rank J(K)$. 
	Suppose that there exists a rational point $P \in X(K)$ (otherwise the conclusion of 
	Theorem \ref{T:mainthm} is trivially true) and let $\iota\colon X \to J$ be the embedding given by  $Q \mapsto [Q-P$]. 
	
 \subsection{Models and Residue Classes}
	Let $\scrX$ be a proper regular model of $X_{\pp}$ over
	$\calO_{K_{\pp}}$ and denote its special fiber by $\scrX_{\pp}$.  We have the following commutative diagram:
	\begin{equation}\label{E:setup}
	\xymatrix{
    	\scrX_{\pp} \ar[r] \ar[d]^{\pi_{\pp}} & 
	\scrX \ar[d]^\pi & X_{\pp} \ar[l]_{i}\ar[d]^{\pi_{K_{\pp}}}&\\
    	\Spec\F_{\pp} \ar[r] & \Spec\OK & \Spec K_{\pp}.\ar[l]
    	}
	\end{equation}

	Since $\scrX$ is proper, the valuative criterion gives a reduction map
	 \begin{equation}\label{E:reduction}
	     \rho\colon X_{\pp}(K_{\pp}) = \scrX(\calO_{K_{\pp}}) \to \scrX(\F_{\pp}).
	 \end{equation}
	 Alternatively, $\rho$ is given by smearing any $K_{\pp}$-point of $X_{\pp}$ to an
	 $\calO_{K_{\pp}}$-point of $\scrX$  and then intersecting with the
	 special fiber $\scrX_{\pp}$; i.e. $\rho(P) = \overline{\{P\}} \cap 
	 \scrX_{\pp}$, where $\overline{\{P\}}$ is the closure of the point $P$ (note that even though $P$ is a closed point of $X_\pp$ it is not closed on $\scrX$). 
Since $\scrX$ is regular, the image is contained in $\scrX_{\pp}^{\text{sm}}(\F_{\pp})$.
Conversely, by Hensel's lemma any point of $\scrX_{\pp}^{\text{sm}}(\F_{\pp})$ lifts to a point in $X_{\pp}(K_{\pp})$.

	  \begin{definition}	     
	     	For  $\widetilde{Q} \in \scrX_{\pp}^{\text{sm}}(\F_{\pp})$ 
	          we define the \defi{residue class}  (or \defi{tube})
	          $D_{\widetilde{Q}}$ to be the preimage $\rho^{-1}\big(\widetilde{Q}\big)$ 	          
		  of $\widetilde{Q}$ under the reduction map \eqref{E:reduction}.     
	 \end{definition}
	 
	 \begin{definition}
       Let $\widetilde{Q} \in \scrX_{\pp}^{\text{sm}}(\F_{\pp})$ and
       let  $\omega \in H^{0}\big(X_{\pp},\Omega^1_{X_{\pp}/K_{\pp}}\big)$.
	     Choose $t \in K_{\pp}^{\times}$ such that $t\omega$ (considered as a rational section of relative dualizing sheaf of $\scrX$) 
	     has neither a pole nor zero along the entire component $X_v$ of
       $\scrX_\pp$ containing $\widetilde{Q}$ 
(in other words, $t\omega$ is allowed to have poles or zeroes at some
individual points of $X_v$, but is not allowed to be identically zero
or infinite when restricted to $X_v$).  
This is possible 
	     because the relative dualizing sheaf is invertible at the generic points of components of the special fiber containing 
	     smooth points.  Let 
	    $\widetilde{\omega}$ be the reduction of $t \omega$ to $X_v$ and define 	     
	     \[         
                 n\big(\omega,\widetilde{Q}\big) = \ord_{\widetilde{Q}} \widetilde{\omega}.
	     \]     
	 \end{definition}

 \subsection{$p$-adic Integration} 
      For a more leisurely introduction to integration on a $p$-adic curve see \cite{PMc_survey}*{Sections 4 and 5}.	
      For $\omega \in H^{0}\big(X_{\pp},\Omega^1_{X_{\pp}/K_{\pp}}\big)$ 
      let $\eta_{\omega}\colon X(K_{\pp}) \to K_{\pp} $ be the function $Q \mapsto \int_P^Q \omega$.
      The following proposition summarizes relevant results of \cite{LT_thue}*{Section 1}.

\newcommand{\Iw}{I_{\omega,Q}(t)}
      \begin{proposition}
	 Let $\widetilde{Q} \in \scrX_{\pp}^{\text{sm}}(\F_{\pp})$ and $Q \in
	 D_{\widetilde{Q}}$. Let $u \in \calO_{\scrX,Q}$ such that the
	 restriction to $\calO_{\scrX_{\pp},\widetilde{Q}}$ is a uniformizer.
	 Then the following are true.
	 \begin{itemize}
	    \item  [(1)] The function $u$ defines a bijection
	         \[
	            u\colon  D_{\widetilde{Q}}\, \widetilde{\to}\, \pp\calO_{K_{\pp}}. 
	         \]	   
	    \item  [(2)] There exists $\Iw \in
		 K_{\pp}[[t]]$ which enjoys the following properties:
		 \begin{itemize}
		     \item  [(i)] For $Q' \in D_{\widetilde{Q}}$, $\eta_{\omega}(Q') = I_{\omega,Q}(u(Q')) + \eta_{\omega}(Q)$.
		 
		     \item  [(ii)] $w(t) := \Iw' \in \calO_{K_{\pp}}[[t]]$. 
		 
		     \item  [(iii)] If we write $w(t) = \sum_{i =
		     0}^{\infty} a_{i} t^{i}$, then 
		     \[
		         \min\left\{i \colon v(a_{i}) = 0\right\} = n(\omega,\widetilde{Q}).
		     \]
		 \end{itemize}
	     \end{itemize}
	 \end{proposition} 

	The starting point of Chabauty's method is the following proposition.

	\begin{proposition}[\cite{stoll_indep}, Section 6]
\label{P:differentials}
	  Denote by $V_{\chab}$ the vector space of all
	  $\omega \in H^{0}\big(X_{\pp},\Omega^1_{X_{\pp}/K_{\pp}}\big)$ 
	  such that $\eta_{\omega}(Q) = 0$ for all $Q \in X(K)$. Then $\dim V_{\chab} \geq g-r$.
	\end{proposition}
	
 \subsection{Newton Polygons} 

     We now will use Newton polygons to bound the number of zeroes of 
$I_{\omega,Q}(u(Q'))$ as $Q'$ ranges over $D_{\widetilde{Q}}$. 
Following \cite{stoll_indep}*{Section 6}, we let $e = v(p)$ be
     the absolute ramification index of $K_{\pp}$ and make the following
     definitions, where $v_p$ is the valuation of $\Q_p$ (and in
     particular, normalized so that $v_p(p) = 1$).
     \begin{definition}
	We set 
	\[
	    \nu(Q) = \#\left\{ Q' \in D_{\widetilde{Q}} \text{ such
          that } I_{\omega,Q}(u(Q')) = 0 \text{ for all } \omega \in V_{\chab} \right\}
	\]
	and 
	\begin{align*}
	   \delta(v, n) 
		&= \max\{ d \ge 0 \mid n+d+1 - v(n+d+1) \le n+1 - v(n+1) \} \\
		&= \max\{ d \ge 0 \mid e\,v_p(n+1) + d \le e\,v_p(n+d+1) \} \,.
	\end{align*}	
    \end{definition}
    The key proposition is \cite{stoll_indep}*{Proposition 6.3} where a Newton polygon
    argument gives the following bound.
   \begin{proposition} 
     \label{P:residueBound}
     We have the bound 
     \[
     	\nu(Q) \leq 1 + n\big(\omega,\widetilde{Q}\big) +
	\delta\big(v,n\big(\omega,\widetilde{Q}\big)\big).
     \]
     Furthermore, suppose $e < p - 1$. Then $\delta(v, n) \le e\,\lfloor n/(p-e-1) \rfloor$.
     In particular, if $p > n + e + 1$, then $\delta(v, n) = 0$.
  \end{proposition}

\subsection{Bounding $\#X(K)$}
\label{ss:main}
    We bound $\#X(K)$ as follows. For each $\widetilde{Q} \in \scrX^{\text{sm}}_{\pp}(\F_{\pp})$, 
$$\#\big(X(K) \cap D_{\widetilde{Q}}\big) \leq \nu(Q).$$ 
    For nonzero $\omega \in V_{\chab}$ (see Definition \ref{P:differentials}) 
    summing the bound of Proposition \ref{P:residueBound} over the residue
    classes of each smooth point gives
    \[
	\#X(K) \leq \#\scrX_{\pp}^{\text{sm}}(\F_{\pp}) + \sum_{\widetilde{Q} \in
	\scrX_{\pp}^{\text{sm}}(\F_{\pp})} \big(n\big(\omega,\widetilde{Q}\big) +
	\delta\big(v,n\big(\omega,\widetilde{Q}\big)\big)\big).
    \]
    To use this we need to bound 
    \[
	\sum_{\widetilde{Q} \in \scrX^{\text{sm}}_{\pp}(\F_{\pp})} n\big(\omega,\widetilde{Q}\big).
    \]	
    As in the good reduction case of \cite{Cole_eff}, Riemann-Roch  
    gives, for a fixed $\omega$, the preliminary
    bound
    \[
	\sum_{\widetilde{Q} \in \scrX_{\pp}^{\text{sm}}(\F_{\pp})} n\big(\omega,\widetilde{Q}\big) \leq 
	\deg \divv \omega
	= 2g - 2.
    \]
    If $p > 2g + e-1$, then in particular
    $p > n(\omega,\widetilde{Q}) + e + 1$ for every $\widetilde{Q}$ and Proposition
    \ref{P:residueBound} reveals that
    $\delta(v,n(\omega,\widetilde{Q})) = 0$,
    recovering the bound of Theorem \ref{T:PMC}
    \[
    \#X(K) \leq \#\scrX_{\pp}^{\text{sm}}(\F_{\pp}) + 2g - 2.
    \]

\subsection{The Chabauty divisor}
\label{ss:chabautyDivisor}

      The idea of the proof of Theorem \ref{T:stollthm} (due to  \cite{stoll_indep}) is to use a different
      differential $\omega_{\widetilde{Q}}$ for each residue class to
      get a better bound. Stoll does this for the good reduction case
      \cite{stoll_indep}*{Theorem 6.4} and what prevents his method
      from working in full generality is that the reduction map
      \begin{equation}\label{e:reduction}
        \P\big(H^0\big(X_{\pp},\Omega^1\big)\big) \to 
        \P\big(H^0\big(\scrX_{\pp},\Omega^1\big)\big)
      \end{equation}
     (and thus the classical rank function $r(D)$ of Definition \ref{d:classicalRank})
      is well behaved only when $\scrX$ is
      smooth. 

The key feature of Stoll's proof is the following: Chabauty's method  produces a divisor $D$ with the property that a bound on $\deg D $ gives a bound on $\#X(K)$. In this subsection we extend Stoll's construction of this divisor to the case of bad reduction.\\

    For a map $f\colon Y \to Z$ denote by $\omega_{f}$ the relative dualizing sheaf of $f$. 
    Recall the setup of \eqref{E:setup}.
    The appropriate generalization of the reduction map \eqref{e:reduction} is the following. 
    Since the structure map $\pi\colon \scrX \to \Spec \calO_{K_{\pp}}$ is flat \cite{liu}*{p. 347}, base change for relative dualizing sheaves
    \cite{liu}*{Theorem 6.4.9} gives an isomorphism $i^{*}\omega_{\pi} \isom
    \omega_{\pi_{K_{\pp}}} \isom \Omega^1_{\Xp/K_{\pp}}$. One gets an inclusion of global sections
     \begin{center}$\xymatrix{
       0 \ar[r] & \homX \ar[r]^-{\phi} & \homX \otimes_{\calO_{K_{\pp}}} K_{\pp}
     \ar@{=}[r]& H^0\big(\Xp,\Omega^1_{\Xp/K_{\pp}}\big)}$.
    \end{center}
    A subspace $V \subset H^0\big(\Xp,\Omega^1_{\Xp/K_{\pp}}\big)$ thus pulls back to a submodule
    $V_{\calO_{K_{\pp}}} := \phi^{-1}(V) \subset \homX$.
    Now let $V = V_{\chab}$ (see Proposition \ref{P:differentials}),
    and for any $\widetilde{Q} \in \scrX_{\pp}^{\text{\text{sm}}}(\F_{\pp})$ let
    \[
        n_{\widetilde{Q}} := \min \big\{ n\big(\omega,\widetilde{Q}\big)  \,\big| \, \omega \in V \big\}.
    \]    
    Then Proposition \ref{P:residueBound} becomes
    \[
        \nu(\widetilde{Q}) \leq 1 + n_{\widetilde{Q}} +
	\delta\big(v,n_{\widetilde{Q}}\big)
    \]         
    and thus
    \[
	\#X(K) \leq \#\scrX_{\pp}^{\text{sm}}(\F_{\pp}) + \sum_{\widetilde{Q} \in
	\scrX_{\pp}^{\text{sm}}(\F_{\pp})} \big(n_{\widetilde{Q}} +
	\delta(v,n_{\widetilde{Q}})\big).
    \]    

    \begin{definition}
   \label{D:chabautyDivisor}
      We define the \defi{Chabauty Divisor} to be the divisor
      \begin{equation}\label{eq:chabautyDivisor}
        D_{\chab,\pp} = \sum_{\widetilde{Q} \in \scrX_{\pp}^{\text{\text{sm}}}(\F_\pp)}
        n_{\widetilde{Q}}\widetilde{Q}.
      \end{equation}
    \end{definition}

Since $\sum n_{\widetilde{Q}} = \deg D_{\chab,\pp}$, the goal is to bound $\deg D_{\chab,\pp}$. \\

In what follows, we introduce new ideas to show that $\deg D_{\chab,\pp} \leq 2r$.

 \section{Geometric Rank functions}
 \label{S:geometricRank}

In this section we define the rank functions $r_{\num}(D) , r_{\Ab}(D)
, r_{\toric}(D) ,  r_X(D)$ and prove Clifford's theorem for $r_{\ab}$.
\vspace{4pt}

Let $\OO_{K_\pp}$ be a DVR with fraction field $K_\pp$ and
perfect, infinite residue field $\F_\pp$
and let  $\scrX$ be a regular \emph{semistable} curve over $\OO_{K_{\pp}}$
with generic fiber $X_{\pp}$ and special fiber $\scrX_\pp$.

\begin{remark}
Note that in this section we make the additional assumptions that
  \textbf{the residue field is infinite} and that $\scrX$ is \textbf{semistable}. 
 With a view toward our intended application -- bounding the degree of the Chabauty Divisor
  of Definition \ref{D:chabautyDivisor} -- we note that the  degree of
  $D_{\chab, \pp}$ can  only \emph{increase} after field extensions.
  Moreover, after a finite extension of $K_\pp$, $\scrX_{K_\pp}$ is
  dominated by a regular semistable model $\scrX'$ of its generic
  fiber, and again the degree of the Chabauty Divisor of $\scrX'$ can
  only grow.
See Section \ref{S:finalTouch} for the precise argument. 

\end{remark}

\subsection{Motivation -- N\'eron models}
\label{ss:motivation}

 Let $\cd_1,\cd_2$ be horizontal divisors on $\scrX$ (in other words, no component of $\cd_i$ is contained in $\scrX_\pp$).  Let $D_1,D_2$ be their generic fibers.   Suppose $\deg(D_1)=\deg(D_2)$ and that $D_1,D_2$ are supported on $K_\pp$-rational points. Since $\scrX$ is regular, the reductions $\rho(\cd_1),\rho(\cd_2)$ are supported on smooth points of $\scrX_\pp$.

We are interested in the following question: {\em are $D_1,D_2$ linearly equivalent}?
We approach this by considering $D_1-D_2$ and finding obstructions for $D_1-D_2=0$ in the Jacobian $J$ of $X_\pp$.  Let $\cj$ be a N\'{e}ron model for the Jacobian.  $D_1-D_2$ induces a section $d\colon\Spec \calO_{K_{\pp}}\rightarrow\cj$ by the N\'{e}ron mapping property.  So our question becomes: {\em is $d$ supported on the closure of the identity in $\cj$?}

We can refine this question by studying the structure of the N\'{e}ron
model of the Jacobian \cite[Sec. 9.5]{BLR},\cite[App. A]{BakerSpec}.
The N\'{e}ron model is constructed from the relative Picard scheme
$\Pic_{\scrX/\OO_{K_{\pp}}}$.  Let $P$ be the open subfunctor of
$\Pic_{\scrX/\calO_{K_{\pp}}}$ consisting of divisors of total degree
$0$.  Let $U$ be the closure of the unit section $\Spec
\calO_{K_{\pp}}\rightarrow P$.  
Then by \cite[Thm 9.5/4]{BLR}, $\cj=P/U$.  The
$\calO_{K_{\pp}}$-points of $U$ can be described in terms of the
components of $\scrX_\pp$.  In fact, they are line bundles of the form
$\OO_{\scrX}(V)$ where $V$ is a divisor of $\scrX$ supported on the special
fiber.  We will call such divisors {\defi twists} (see Subsection \ref{ss:dualGraphs}). 
These line bundles restrict to trivial bundles on the generic fiber.  

The identity component of the N\'{e}ron model can be related to the Picard scheme.  By \cite[Theorem 9.5.4b]{BLR}, $\cj^0=\Pic^0_{\scrX/\calO_{K_{\pp}}}$ where the latter scheme denotes the subscheme of the Picard scheme consisting of line bundles whose restriction to every component of the special fiber has degree $0$.   Let $\pi\colon\widetilde{\scrX}_\pp\rightarrow \scrX_\pp$ be the normalization.  Then there is an exact sequence \cite[Thm 7.5.19]{liu}
\[0\rightarrow T\rightarrow \Pic^0_{\scrX_\pp}\rightarrow \Pic^0_{\widetilde{\scrX}_\pp}\rightarrow 0\]
where $T$ is a toric group (i.e. geometrically isomorphic to a power of $\G_m$).  Note that $\Pic^0_{\widetilde{\scrX}_\pp}$ is an Abelian variety.  We can use this structure of $\cj$ to give a hierarchy of conditions required for $D_1-D_2$ to be the trivial divisor in $J$:
\begin{enumerate}
\item $\rho(\cd_1-\cd_2)$ must be in the identity component of $\cj_\pp$;

\item The image of $\rho(\cd_1-\cd_2)$ in $\Pic^0_{\widetilde{\scrX}_\pp}$ must be the identity;

\item $\cd_1-\cd_2$ must be the identity in $\cj_\pp$;

\item $\cd_1-\cd_2$ must extend as the identity in $\cj$.
\end{enumerate}

We call the first three steps in the hierarchy {\em numeric, Abelian, toric.} We can rephrase the steps in the hierarchy as finding a better and better approximation to a section of $\OO_{\scrX}(\cd_1-\cd_2)$:
\begin{enumerate}
\item[(1)] There is a twist $\OO_\scrX(V)$ such that $\OO_\scrX(\cd_1-\cd_2+V)$ has degree $0$ on every component of $\scrX_\pp$;

\item[(2)] There is a twist $\OO_\scrX(V)$ together with a section $s_v$ of $\OO_\scrX(\cd_1-\cd_2+V)|_{X_v}$ for each component $X_v$ of $\scrX_\pp$;

\item[(3)] The sections can be chosen above to agree across nodes;

\item[(4)] The section $s$ formed from the $s_v$'s, considered as a section of $\OO_\scrX(\cd_1-\cd_2+V)|_{\scrX_\pp}$ extends to $\scrX$.
\end{enumerate}



Note that steps (1)-(3) are only concerned with $\rho(\cd_1),\rho(\cd_2)$.  Step (1) has been studied in the context of linear systems on graphs by Baker-Norine.  Baker-Norine's theory works best when all components of $\scrX_\pp$ are rational.  In this case step (2) imposes no constraints. 
See Subsection \ref{S:bakerNorineBackground} for a short review.
\vspace{4pt}

 In this paper, we will advocate for including step (2) when not all
 components of the special fiber are rational. 

\subsection{Dual graphs}
\label{ss:dualGraphs}

Let $\scrX / \calO_{K_{\pp}}$ be a regular, semistable curve.  We define the \defi{dual
  graph} $\Gamma$ of $\scrX$ to be the graph whose vertices are the irreducible components $X_v$ of $\scrX_\pp$ and with an edge between the vertices corresponding to nodes of $\scrX_\pp$.

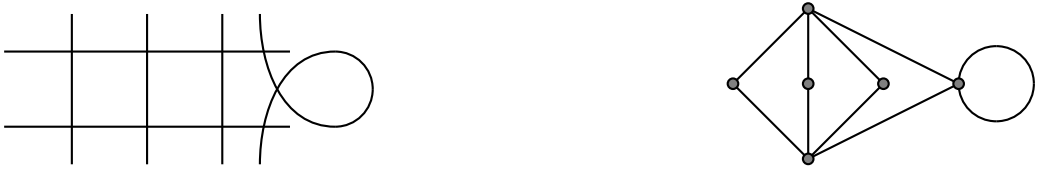
\begin{figure}
\begin{center}
\tikzstyle{every node}=[circle, draw, fill=black!50,
                        inner sep=0pt, minimum width=4pt]
   \begin{tikzpicture}
    \draw[thick] (-1.9,.5) -- (1.9,.5);
    \draw[thick] (-1.9,-.5) -- (1.9,-.5);
    \draw[thick] (-1,1) -- (-1,-1);
    \draw[thick] (-0,1) -- (0,-1);
    \draw[thick] (1,1) -- (1,-1);

  \draw [thick] (1.5,1) to [out=-90,in=-180] (2.5,-.5);
  \draw [thick] (1.5,-1) to [out=90,in=-180] (2.5,.5);
  \draw [thick] (2.5,-.5) to [out=0,in=-90]	(3,0);
  \draw [thick] (2.5,.5) to [out=0,in=90]	(3,0); 
   \end{tikzpicture}
\quad \hspace{4cm}
\begin{tikzpicture}[thick]
\foreach \x in {0,1,2}
{\draw (\x,0)  -- (1,-1);
\draw (\x,0) -- (1,1);
}
\draw (3,0) -- (1,1);
\draw (3,0) -- (1,-1);

\foreach \x in {3.5}
{  \draw [thick] (\x + -.5,0) to [out=-90,in=180] (\x 0,-.5);
  \draw [thick] (\x + 0,-.5) to [out=0,in=-90] (\x + .5,0);
  \draw [thick] (\x + .5,0) to [out=90,in=0] (\x + 0,.5);
  \draw [thick] (\x + 0,.5) to [out=180,in=90] (\x + -.5,0);
}
\foreach \x in {0,1,2,3} 
{\draw [fill=gray] (\x,0) circle (2pt);}
 \draw [fill=gray] (1,-1) circle (2pt); 
 \draw [fill=gray] (1,1)   circle (2pt);

\end{tikzpicture}
\end{center}
\caption{A curve and its dual graph.}
\label{F:C5 mod 2}
\bigskip
\hrule
\end{figure}

 \begin{remark}
  Our usage differs from that of Baker \cite{BakerSpec} for the following reasons: we do not require the models to be strictly semistable; vertices in the dual graph correspond to irreducible
  components of $\scrX_\pp$, not to components of the normalization.  Consequently, components $X_v$ of $\scrX_\pp$ may have nodes. 
  
   Because there are multiple edges in our graph, we will treat edges as a multiset in $\Sym^2 V(\Gamma)$, so there may be many edges $e$ satisfying $e=vw$ for a pair of vertices $v,w$.
\end{remark}

Let $\calM(\Gamma)$ be the Abelian group of integer-valued functions on the vertices of $\Gamma$.  Let
$\Div(\Gamma)$ be the free Abelian group generated
by the vertices of $\Gamma$.   For $D=\sum_v a_vv\in\Div(\Gamma)$, write $D(v)$ for the coefficient $a_v$.  We say that $D$ is effective and write $D\geq 0$ if $D(v)\geq 0$ for all $v\in V(\Gamma)$.  The \defi{Laplacian} on $\Gamma$ is defined to be the map
\begin{eqnarray*}
\Delta\colon\calM(\Gamma)&\rightarrow&\Div(\Gamma)\\
\varphi &\mapsto& \sum_v\left(\sum_{e=wv} \left(\varphi(w)-\varphi(v)\right)\right)\left(v\right).
\end{eqnarray*}
Let $\calM(\Gamma)_\sim$ be the quotient of $\calM(\Gamma)$ by
constant functions.  Note that $\Delta$ is defined on
$\calM(\Gamma)_\sim$.  

We can view $\varphi\in \calM(\Gamma)$ as a
\defi{twist} (i.e. a vertical divisor on $\scrX$) by setting 
$V_{\varphi} = \sum_v \varphi(v)X_v$.   
Since the special fiber of $\scrX$ is linearly equivalent to $0$,
elements of $\calM(\Gamma)$ that differ by a constant function
induce linearly equivalent twists. For a divisor $\calD$ on $\scrX$,
we define the \defi{twist} of $\calO_{\scrX}(\calD)$ by $\varphi$ to be
$\calO_{\scrX}(\calD)(\varphi) =\calO_{\scrX}(\calD+V_\varphi)$. 
\vspace{4pt}

Let $\Div(X_\pp(K_{\pp}))$ be the divisors of $X_\pp$ supported on
$K_{\pp}$-points (i.e. each individual point is defined over $K_\pp$).
There is a \defi{reduction} map
\[\rho\colon X_\pp(K_{\pp})\rightarrow \scrX_\pp(\F_\pp)\]
given by $P\mapsto\overline{\{P\}}\cap \scrX_\pp$.
By regularity of $\scrX$, this map is surjective onto $\scrX_{\pp}^{\sm}(\F_\pp)$ and therefore induces a surjective map
\[\rho\colon\Div(X_\pp(K_{\pp}))\rightarrow\Div(\scrX_{\pp}^{\sm}(\F_\pp)).\]
There is a \defi{multidegree} map
\begin{eqnarray*}
\mdeg\colon \Div(X_\pp(K_{\pp}))&\rightarrow &\Div(\Gamma)\\
D&\mapsto& (v\mapsto \deg(\OO_\scrX(D)|_{X_v})).
\end{eqnarray*}
that factors through $\rho$.  Note that our usage differs from that of \cite{BakerSpec} in that what we call the multidegree map is there called the specialization map.
The canonical bundle $K_{\scrX_\pp}$ of the special fiber has the following multidegree:
\[\mdeg(K_{\scrX_\pp})=\sum_v (2g(X_v)+\deg(v)-2)(v).\]

Since $\scrX$ is proper, $X_\pp(K_\pp) = \scrX(\calO_{K_\pp})$ and
thus $\Div X_\pp(K_\pp)$ is naturally identified with horizontal
divisors in $\Div \scrX(\calO_{K_\pp})$. In particular, we will freely
pass between divisors supported on $X_\pp(K_\pp)$ and their
closures.  For example, for $\varphi \in \calM(\Gamma)$ and $D \in \Div
X_\pp(K_\pp)$, the twist $D(\varphi)$ refers to the divisor $\calD +
V_{\varphi}$, where $\calD$ is the closure of $D$. Moreover, we will apply $\rho$ and $\mdeg$ to horizontal
divisors.

\subsection{Rank functions}
\label{ss:rankFunctions}

We can use the steps in the hierarchy of Subsection \ref{ss:motivation} to define geometric rank functions.   Let $D$ be a divisor in $\Div(X_\pp(K_{\pp}))$ and $E$ be an effective divisor in $\Div(X_\pp(K_{\pp}))$.  Let $\cd,\ce$ denote the closures of $D,E$ on $\scrX$. We consider the following conditions:
\begin{enumerate}
\item[(1)] There is a twist $\varphi$ such that $\OO_\scrX(\cd-\ce)(\varphi)$ has non-negative degree on every component of $\scrX_{\pp}$;

\item[(2)] There is a regular section $s_v$ of $\OO_\scrX(\cd-\ce)(\varphi)|_{X_v}$ for each $v\in V(\Gamma)$;

\item[(3)] The sections $s_v$ can be chosen to agree across nodes;

\item[(4)] The section $s$ formed from the $s_v$'s  extends to $\scrX$ as a section of $\OO_\scrX(\cd-\ce)(\varphi)$.
\end{enumerate}

\begin{definition}
\label{R:rankFunctions}
  We define \defi{rank functions} $r_1,r_2,r_3,r_4$, which we denote by
  $r_{\num}, r_{\Ab}, r_{\toric}, r_X$, as follows. Let $D\in\Div(X_\pp(K_\pp))$. 
 We say
  $r_i(D)=-1$ if and only if $D$ fails to satisfy steps $1$ to $i$, and for $r \geq 0$,
 we say  $r_i(D)=r$ if and only if for any effective divisor $E$ of degree $r$ on $X_\pp(K_\pp)$, $D-E$ satisfies steps $1$ to $i$ and there is an effective divisor $E$ on $X_\pp(K_\pp)$ of degree $r+1$ for which $D-E$ fails
  to satisfy one of steps $1$ to $i$.  

\end{definition}

\begin{definition} We say two divisors $D,E\in\Div(X_\pp(K_\pp))$ of the same degree are ($\num,\Ab,\toric$)\defi{-linearly equivalent} if the difference of their closures on $\scrX$,
$\cd-\ce$, satisfies steps $1$ to $1,2$, or $3$ respectively.
\end{definition}

These rank functions do not all depend on the geometry of the generic fiber $X_\pp$.  Instead, some of them capture the combinatorics and the geometry of the special fiber.  First, we note that the restriction of a twist  to the special fiber depends only on the special fiber.

\begin{lemma} 
\label{l:centralfiber}
Let $\varphi
\in\calM(\Gamma)$, considered as the twist $\sum \varphi(v) X_v$. 
Let $X_{v_0}$ be a component of the special fiber.  For an edge $e$ of $\Gamma$ at $v_0$, let $p_e$ be the corresponding node on $X_{v_0}$.   Then,
\[\OO_\scrX(\varphi)|_{X_{v_0}}=\sum_{e=v_0w} (\varphi(w)-\varphi(v_0))(p_e)\]
where the sum is over the edges at $v_0$.
Moreover, 
\[\mdeg(\OO_{\scrX}(\varphi))=\Delta(\varphi).\]

\end{lemma}

\begin{proof}
Because the special fiber of $\scrX$ is linearly equivalent to $0$ on $\scrX$, $\varphi$ is linearly equivalent to $\sum (\varphi(v)-\varphi(v_0)) X_{v}$.  The lemma follows from the fact that
\[\OO_\scrX(X_w)|_{X_{v_0}}=\sum_{e=v_0w} p_e,\]
and the claim about the multidegree of $\OO_\scrX(\varphi)$ now follows from definitions.
\end{proof}

\begin{remark}
  Because steps (1)-(3) only depend on the reduction of $E$ and the reduction map is surjective, to
  determine the rank, one need only look at classes of effective
  divisors in $\Div(\scrX_\pp^{\sm}(\F_\pp))$.  Furthermore,
  $r_{\num}$ is only sensitive to the dual graph $\Gamma$ and the
  multidegree $\mdeg(\cd)$ of $\cd$ and is equal to the rank function of
  Baker-Norine \cite{BakerN:RR} on the dual graph of $\scrX_\pp$.
  Because the rank functions $r_{\Ab},r_{\toric}$ only depend on
  $\rho(\cd)$, they factor through the reduction map.  We may therefore
  view $r_{\Ab},r_{\toric}$ as defined on
  $\Div(\scrX_\pp^{\sm}(\F_\pp))$.  
\end{remark}


\begin{remark}
One can also study the ranks associated to a regular proper minimal model.  In this case, the Jacobian will have unipotent factors coming form non-reduced components. The combinatorics of bad reduction in this case have been studied by Lorenzini \cite{L_dual,L_zeta} and are quite involved.
\end{remark}

 \subsection{Rank of a divisor on a graph}
 \label{S:bakerNorineBackground}

We review some facts from the theory of linear systems on graphs
\cite{BakerSpec, BakerN:RR}, and in particular the identification of $r_{\num}$ with the rank of a divisor on a graph.\vspace{4pt}
 
\begin{definition}

Two divisors $D_1,D_2$ on a graph are said to be linearly equivalent if $D_1-D_2=\Delta(\varphi)$ for some $\varphi\in\calM(\Gamma)$.
For a non-negative integer $r$, a divisor $D$ on $\Gamma$ is said to have \defi{rank at least} $r$ if for
any effective divisor $E$ on $\Gamma$ of degree $r$, the divisor $D-E$ is linearly equivalent to an
effective divisor. We define the \defi{rank} $r(D)$ of $D$ 
to be $-1$ if $D$ is not linearly equivalent to an effective divisor, and otherwise we define $r(D)$ to be 
the
largest integer $r$ such that $D$ has rank at least
$r$. 
\end{definition}

\begin{remark}
Let $\scrX / \calO_{K_{\pp}}$ be a regular, semistable curve and let $D \in \Div X_\pp(K_{\pp})$ be a divisor. Then it follows from Lemma \ref{l:centralfiber} that
$$r_{\num}(D) = r(\mdeg(D)).$$ 
\end{remark}

The \defi{canonical divisor} of $\Gamma$ is
the divisor 
\[K_\Gamma=\sum_v (\deg(v)-2)(v).\]
The degree of this divisor is given by 
\[\deg K_{\Gamma} = 2g(\Gamma)-2\]
where $g(\Gamma)$ is the first Betti number of the graph. 
We have the following comparison between the multidegree of the canonical bundle of $\scrX_{\pp}$ and the canonical divisor on $\Gamma$:
\[\mdeg(K_{\scrX_{\pp}})=K_\Gamma+\sum_v 2g(X_v).\]
This formula follows by the arguments of \cite{BakerSpec}*{Lemma 4.15}.
If $\scrX_{\pp}$ is totally degenerate (i.e. $g(X_v) = 0$ for every
component $X_v$ of $\scrX_{\pp}$), then we do have $\mdeg(K_{\scrX_{\pp}}) = K_\Gamma$. \\


We say that a divisor $D$ on $\Gamma$ is \defi{special} if $r(K_{\Gamma} - D) \geq 0$. We will make extensive use of the following:
\begin{theorem}[Clifford's theorem for graphs
  \cite{BakerN:RR}*{Corollary 3.5}]
\label{T:cliffordGraph}
  Let $\Gamma$ be a graph and let $D \in \Div(\Gamma)$ be an effective special divisor. Then
\[r(D)\leq \frac{1}{2}\deg D.\]
\end{theorem}

When $\scrX_\pp$ is totally degenerate, this theorem is sufficient for the
proof of Theorem \ref{T:mainthm}. In general
though, the dual graph $\Gamma$ of $\scrX_{\pp}$ may have
\emph{smaller} genus than $\scrX_{\pp}$, and the bound on the degree $\deg
D_{\chab,\pp}$ of the Chabauty divisor will be $2r + 2(g(\scrX_{\pp}) - g(\Gamma)) >
2r$, which is not sufficient for the application to the
Chabauty-Coleman bound.

\subsection{Upper bounds}
\label{ss:upperBounds}

We now observe that these rank functions provide upper bounds for the
dimensions of linear subspaces on $X_\pp$.  For $r_{\num}$, this is the specialization lemma of Baker \cite{BakerSpec}.


\begin{proposition} (Compare \cite[Lemma 2.4]{BakerSpec}) 
\label{L:upperBounds}  Suppose that $X_\pp(K_\pp)$ is infinite.  
Let $D$ be a divisor in $\Div(X_\pp(K_\pp))$.  Then $r_X(D)=h^0(X_\pp,\calO_{X_\pp}(D))-1$.
\end{proposition}

\begin{proof}
Let $r=h^0(X_\pp,\calO_{X_\pp}(D))-1$.  Let $E\in \Div(X_\pp(K_\pp))$ be an effective divisor of degree $r$.  Because the 
condition that a section of $\calO_{X_\pp}(D)$ vanishes at $E$ imposes at most $r$ conditions, one can 
find a section $s$ on $X_\pp$ vanishing at $E$.  Let $\cd=\overline{D}$ and extend the section $s$ as 
a rational section of $\OO_\scrX(\calD)$.   Decompose the divisor of $s$ as the sum of a 
horizontal divisor $H$ and a vertical divisor $V =  -\sum_v \varphi(v)X_v$ for some $\varphi\in\calM(\Gamma)$.  If $s$ is considered as a section of 
$\OO_\scrX(\cd)(\varphi)$, then it is regular section and does not vanish along any component of the 
special fiber.   Using that convention, we set $s_v=s|_{X_v}$  for each component $X_v$ of the special fiber.  Then the $s_v$'s satisfy conditions (1)-(4). We conclude that $r_X(D) \geq r$.

Now we show $r_X(D)\leq r$.  Because $X_\pp(K_\pp)$ is infinite, there is a divisor $E\in \Div(X_\pp(K_\pp))$ of degree $r+1$ such that there is no non-zero section of $\OO_{X_\pp}(D)$ vanishing on $E$.  If $r_X(D)\geq r+1$, then there would be a twist $\varphi$ such that $\OO_\scrX(\cd-\ce)(\varphi)$ has a section $s$ that is non-zero on every component of $\scrX_\pp$.  Such a section $s$ restricts to the generic fiber as a non-zero section of $\calO_{X_\pp}(D-E)$.  This is impossible.
\end{proof}

\begin{remark}
If the residue field $\F_\pp$ is algebraically closed, then $\scrX_{\pp}^{\text{sm}}(\F_{\pp})$ has infinitely many points.  By the surjectivity of the reduction map, $X_\pp(K_\pp)$ is infinite and the above proposition applies.
\end{remark}

Since, by Definition \ref{R:rankFunctions}, $r_1(D)\geq r_2(D)\geq r_3(D)\geq r_4(D)$, we have the following corollary:

\begin{corollary} 
\label{c:rankUpperBound}
Suppose that $\F_\pp$ is algebraically closed.
Let $D$ be a divisor in $\Div(X_\pp(K_\pp))$.  Then $r_{\num}(\rho(D)),r_{\Ab}(\rho(D)),r_{\toric}(\rho(D))$ provide upper bounds for $h^0(X_\pp,\OO_{X_\pp}(D))-1$. 
\end{corollary}

\begin{remark}
\label{R:rankStrata}
 There are thus inequalities $r_{\num}(D) \geq  r_{\Ab}(D) \geq  r_{\toric}(D) \geq  r_X(D)$; see Section \ref{S:examples} for examples where the inequalities are strict.
\end{remark}

\begin{remark}
  The rank functions $r_{\Ab}$ and $r_{\toric}$ are sensitive to unramified field extensions.
  When one extends the residue field, one allows more choices for the effective divisor
  $E$, so the rank functions do drop with field extensions.  This is not
  a problem since they do stabilize at a finite extension.  The proof
  of that requires results from the next subsection.  To avoid this
  issue, we can pass to the maximal unramified field extensions.  This
  does not affect the specialization lemma or the  rank $r_X$.  In fact, extensions of the residue field give better bounds for $r_X$.
  \end{remark}

\subsection{Twist General Position}
Recall that the residue field of $\OO_{K_\pp}$ is
infinite.  All divisors $D_\pp$ on $\scrX_\pp$ will be supported at smooth
points.

\begin{lemma} 
\label{l:finiteness} 
Let $D$ be a divisor on $\Gamma$.  Then the set of twists that make $D$ effective,
\[S_D=\{\varphi\in\calM(\Gamma)_\sim\ |\ \Delta(\varphi)+D\geq 0\}\]
is finite.
\end{lemma}


\begin{proof}
Given $\varphi\in \calM(\Gamma)$, we may translate to suppose $\max(\varphi(v))=0$.  We
show that $\varphi$ is bounded below by a constant in terms of $\mdeg(D_\pp)$.
Suppose $\max(\varphi)$ is achieved at some vertex $v_0$.
Then we have
\[\sum_{e=wv} (\varphi(w)-\varphi(v)) + D(v) \geq 0.\]  
for all $v$.  For $w_0$ adjacent to $v$, we have
\[\varphi(w_0)
\geq \varphi(v) - D(v)+\sum_{\substack{e=wv\\w\neq w_0}} (\varphi(v)-\varphi(w))
\geq \deg(v)\varphi(v)-D(v).\]
Taking $v = v_0$ and continuing outward, we see that $\varphi$ is bounded below.  
\end{proof}



\begin{definition} Let $D_\pp$ be a divisor on $\scrX_\pp^{\sm}$.  An effective divisor $E_\pp$ is said to be in \defi{twist general position} with respect to $D_\pp$ if for all components $X_v$ and all twists 
$\varphi\in \calM(\Gamma)$
with $\mdeg(\OO_{\scrX_\pp}(D_\pp)(\varphi))\geq 0$, 
\[h^0(\OO_{\scrX_\pp}(D_\pp-E_\pp)(\varphi)|_{X_v})=\max(0,h^0(\OO_{\scrX_\pp}(D_\pp)(\varphi)|_{X_v})-\deg(\OO_{\scrX_\pp}(E_\pp)|_{X_v})).\]
\end{definition}

In other words $E_\pp$ imposes $\deg(E_\pp|_{X_v})$ conditions on
sections of $\OO_{\scrX_\pp}(D_\pp)|_{X_v}$.  There exist twist
general position divisors of any multidegree. Note also that to make
sense of the expression $\mdeg(\OO_{\scrX_\pp}(D_\pp)(\varphi))$, it
suffices to pick a lift $\calD$ of $D_\pp$ and consider $\mdeg(\OO_{\scrX}(\calD)(\varphi))$ 

\begin{lemma} Let $D_\pp$ be a divisor on $\scrX_\pp$.  Let
  $E\in\Div(\Gamma)$.  Then there exists a divisor
  $E_\pp\in\Div(\scrX_\pp^{\text{sm}}(\F_\pp))$ such that
  $\mdeg(E_\pp)=E$ and $E_\pp$ is in twist general position with
  respect to $D_\pp$.
\end{lemma}

\begin{proof}
For any twist $\varphi$ and component $X_v$, the set of
effective divisors $F_v$ supported on $X_v \cap \scrX_{\pp}^{\sm}$ of degree $E(v)$ with 
\[h^0(\OO_{X_\pp}(D_\pp)(\varphi)|_{X_v}-F_v)=\max(0,h^0(\OO_{X_\pp}(D_\pp)(\varphi)|_{X_v})-E(v))\]
is non-empty and Zariski open in $\Sym^{E(v)}(X_v^{\sm})$.  Call this set $U_{\varphi,v}$.  
Now let $U_v=\cap_\varphi U_{\varphi,v}$ where the intersection is over the finite set constructed in Lemma \ref{l:finiteness}.  
Choose a divisor $E_v$ in $U_v$ for each $v$.  The sum of divisors, $E_\pp=\sum E_v$ is the desired divisor.
\end{proof}

\subsection{Clifford Bounds}

We prove the Clifford bound for $r_{\Ab}$.  By Remark
\ref{R:rankStrata} this implies the corresponding bounds for
$r_{\toric}$.  We suppose that $\F_\pp$ is algebraically closed.  We
also pick once and for all a (not necessarily effective) divisor
$K_{\scrX_\pp}\in\Div(\scrX_\pp^{\sm}(\F_\pp))$ in the class of the
canonical bundle on the special fiber. 

Such a divisor exists by Bertini's theorem; indeed, by
\cite{liu}*{7.1, Lemma 1.31} any line
bundle on a projective variety is a difference of very ample line
bundles $\scrL$ and $\scrL'$, and by Bertini's theorem
\cite{hart}*{Theorem 8.18} we
can choose a section of each which does not vanish at the finitely many singular points of $\scrX_\pp$.

\begin{definition}
  Let $L$ be a line bundle on $\scrX_\pp$.  For a vertex $v$ of the dual
  graph $\Gamma$, we say $v$ is \defi{Clifford with respect to} $L$ if
  $\deg(L|_{X_v})<2g$. 
\end{definition}
 In this case, the usual Clifford bounds (if $L$ is special) or the Riemann-Roch theorem (if $L$ is non-special) imply
  \[h^0(L|_{X_v})<g+1.\]
  Consequently, for a generic degree $g$ divisor $D$ on $X_v$, we do not expect $h^0(L|_{X_v}-D)$ to have a non-zero section.

\begin{proposition} 
\label{p:clifford} 
Let $D_\pp$ be an effective divisor on $\scrX_\pp$
supported at smooth points.  Then,
\[r_{\Ab}(K_{\scrX_\pp}-D_\pp)\leq g-\frac{\deg{D_\pp}}{2}-1.\]
\end{proposition}

\begin{remark}
  The idea of the proof is the following. For a vertex $v$
  corresponding to a component $X_v$ of $\scrX_\pp$,
\[ \deg (K_{\scrX_\pp})|_{X_v} = K_{\Gamma}(v) + 2g(X_v);\]
in particular, if $K_\Gamma-\mdeg(D_\pp)$ is not equivalent to an
effective divisor on $\Gamma$ (e.g., it may have negative degree), then for every
twist $\varphi$, there is a component $X_{v_\varphi}$ of $\scrX_\pp$
such that
\[\deg((K_{\scrX_\pp}-D_\pp)(\varphi)|_{X_{v_\varphi}}) < 2g(X_{v_\varphi}).\]
In particular, Clifford's theorem or the Riemann-Roch theorem for $X_{v_\varphi}$ implies that
\[h^0((K_{\scrX_\pp}-D_\pp)(\varphi)|_{X_{v_\varphi}}) < g(X_{v_\varphi})+1.\] The
upshot is that, if  $Q$ is a divisor in twist general position
such that $\deg Q|_{X_v} = g(X_v)$ for all $v$, then for every twist $\varphi$, $(K_{\scrX_\pp}-D_\pp- Q)(\varphi)|_{X_{v_\varphi}}$ has no sections,
so $r_{\ab}(K_{\scrX_\pp}-D_\pp) < \deg Q = \sum g(X_v)$.
\vspace{3pt}

This does not quite give the desired bound, since it may happen that
$\sum g(X_v) > g-\frac{\deg{D_\pp}}{2}-1$; the subtleties of the proof are:
\begin{itemize}
\item [1)] to find a component $X_v$ of $\scrX_\pp$ such that
  fewer than $g(X_v)$ points are necessary to force the rank of 
  $(K_{\scrX_\pp}-D_\pp- Q)(\varphi)|_{X_{v_\varphi}}$ to be $-1$, 
and
\item [2)] when $K_{\scrX_\pp}-D_\pp$ has sections on every
  component, we can add a divisor $Q'$ in twist general position 
   such that $K_{\scrX_\pp}-D_\pp - Q'$ has no
  sections and apply (1); we can keep the degree of $Q'$ small by
  applying the Riemann-Roch theorem for graphs to  $K_\Gamma-\mdeg(D_\pp)$.
\end{itemize}

\end{remark}

\begin{proof}[Proof of Proposition \ref{p:clifford}]
We first consider the case $\deg(K_\Gamma-\mdeg(D_\pp))<0$.  
Let $S_{\mdeg(D_\pp)}$ be the finite set of Lemma \ref{l:finiteness}.
If $S_{\mdeg(D_\pp)}$ is empty, the rank of $D_\pp$ is $-1$ and we are done.  So suppose $S_{\mdeg(D_\pp)}$ is non-empty.  For $\varphi\in S_{\mdeg(D_\pp)}$, let $\Cliff(\varphi)$ be the set of vertices
that are Clifford with respect to $(K_{\scrX_\pp}-D_\pp)(\varphi)$.  We call such vertices $\varphi$-Clifford.  By degree considerations, 
$\Cliff(\varphi)\neq \emptyset$ for all $\varphi\in S_{\mdeg(D_\pp)}$.   Let $m$ be the minimum of $\left|\Cliff(\varphi)\right|$ taken over all $\varphi\in S_{\mdeg(D_\pp)}$ and 
\[S_m=\{\varphi\in S_{\mdeg(D_\pp)}\,|\,\left|\Cliff(\varphi)\right|=m\}.\]
Now, let 
\[M(\varphi)=\max\{\deg((K_{\scrX_\pp}-D_\pp)(\varphi)|_{X_v})\,|\,v\in\Cliff(\varphi)\}.\]
Pick $\varphi\in S_m$ maximizing $M(\varphi)$.  Let $v_\varphi$ be a $\varphi$-Clifford vertex such that 
\[\deg((K_{\scrX_\pp}-D_\pp)(\varphi)|_{X_{v_\varphi}})=M(\varphi).\]
Now let $Q_\pp$ be a divisor consisting of 
$\lfloor \frac{M(\varphi)}{2}+1\rfloor$ points on $X_{v_\varphi}$ and $g(X_w)$
points on $X_w$ for each non-$\varphi$-Clifford $w$,  chosen in twist general position with respect to $K_{\scrX_\pp}-D_\pp$.  Note that $\deg(Q)\leq g-\frac{\deg{D_\pp}}{2}$.  

We claim that $r_{\Ab}(K_{\scrX_\pp}-D_\pp-Q_\pp)<0$.  Suppose to the contrary that there is a twist $\varphi'$ such that $(K_{\scrX_\pp}-D_\pp-Q_\pp)(\varphi')$ has a section on every component.  By construction $K_{\scrX_\pp}-D_\pp$ has at least $m$ $\varphi'$-Clifford vertices.  If it has more than $m$ $\varphi'$-Clifford vertices, then there is some vertex $w$ that is $\varphi'$-Clifford but not $\varphi$-Clifford.  Therefore, $Q_\pp$ contains $g(X_w)$ points on $X_w$.  Since $h^0((K_{\scrX_\pp}-D_\pp)(\varphi')|_{X_w})<g(X_w)+1$ by Clifford's theorem or the Riemann-Roch theorem on $X_w$, we can conclude that $h^0((K_{\scrX_\pp}-D_\pp-Q_\pp)(\varphi')|_{X_w})=0$.

We may suppose that $\Cliff(\varphi')=m$.  Now, if $v_{\varphi}$ is $\varphi'$-Clifford, then 
\[\deg((K_{\scrX_\pp}-D_\pp)(\varphi')|_{X_{v_\varphi}})\leq M(\varphi')\leq M(\varphi)\]
by construction. But then by Clifford's theorem
\[h^0((K_{\scrX_\pp}-D_\pp)(\varphi')|_{X_{v_\varphi}})\leq\frac{M(\varphi)}{2}+1\]
so
\[h^0((K_{\scrX_\pp}-D_\pp-Q)(\varphi')|_{X_{v_\varphi}})=0\]
Therefore, we may suppose that $v_\varphi$ is not $\varphi'$-Clifford.  But because $\left|\Cliff(\varphi')\right|=\left|\Cliff(\varphi)\right|$, there is some vertex $w$ that is $\varphi'$-Clifford but not $\varphi$-Clifford.  By the reasoning above, there cannot be a section of $(K_{\scrX_\pp}-D_\pp-Q)(\varphi')$ on that component. \\

We now consider the case $\deg(K_\Gamma-\mdeg(D_\pp))\geq 0$.  By
Baker-Norine's Clifford bounds (Theorem \ref{T:cliffordGraph}),
\[r_{\num}(K_\Gamma-\rho(D_\pp))\leq g(\Gamma) 
-\frac{\deg{D_\pp}}{2}-1.\]
Therefore, if we set $r=\lfloor g(\Gamma) 
-\frac{\deg{D_\pp}}{2}-1\rfloor$,
we may pick a divisor $Q'\in\Div(\Gamma)$ of degree $r+1$ such that 
$K_\Gamma-\mdeg(D_\pp)-Q'$ is not $\num$-linearly equivalent to an effective divisor on $\Gamma$.   Now, pick $Q'_\pp\in\Div(\scrX_\pp^{\text{sm}}(\F_\pp))$ with $\mdeg(Q'_\pp)=Q'$.
Now, for each $v$, let $Q_v\in\Div(\scrX_\pp^{\text{sm}}(\F_\pp))$ be a divisor of degree $g(v)$ in twist general position with respect to $K_{\scrX_\pp}-D_\pp-Q'_\pp$.  Let 
\[Q_\pp=Q'_\pp+\sum_v Q_v.\]
Note 
\[\deg(Q_\pp)=r+1+\sum g(X_v)=\lfloor g(\Gamma) 
+\sum g(X_v)-\frac{\deg{D_\pp}}{2}-1\rfloor+1\leq g-\frac{\deg(D_\pp)}{2}.\]
We claim that $r_{\Ab}(K_{\scrX_\pp}-D_\pp-Q_\pp)<0$.  In other words, even after twisting by some $\varphi$, there is some component $X_v$ such that $K_{\scrX_\pp}-D_\pp-Q_\pp|_{X_v}$ has no section.  By the definition of Baker-Norine rank, for any twist $\varphi$ there is a vertex $v$ such that $(K_\Gamma-\mdeg(D_\pp)-\mdeg(Q_\pp))(v)<0$.  Consequently, 
\[\deg((K_{\scrX_\pp}-D_\pp-Q'_\pp)(\varphi)|_{X_v})<2g\]
and the usual Clifford or Riemann-Roch bounds apply to show that 
\[h^0((K_{\scrX_\pp}-D_\pp-Q'_\pp)(\varphi)|_{X_v})<g+1.\]
By our choice of $Q_\pp$ and the definition of twist general position:
\[h^0((K_{\scrX_\pp}-D_\pp-Q_\pp)(\varphi)|_{X_v})=0.\]
\end{proof}

\begin{remark}
\label{r:clifford}
It follows from the proof that one can in fact choose the divisor $E_\pp$
to avoid any finite set of points on $\scrX_\pp$. In other words,
there exists an effective divisor $E_\pp \in \Div(\scrX_\pp^{\sm}(\F_\pp))$  such that
\begin{itemize}
\item [(i)]   $\deg E_\pp \leq \frac{1}{2}\deg(K_{X_\pp}-D_\pp)$,  
\item [(ii)] For any twist $\varphi$, $\calO_{\scrX}(\varphi)|_{X_v}\otimes \calO_{\scrX_\pp}(K_{\scrX_\pp}-D_\pp - E_\pp)|_{X_v}$ has no non-zero sections for some component $X_v$ of $\scrX_\pp$, and
\item [(iii)] $\Supp(E_\pp) \cap \Supp(K_{\scrX_\pp}-D_\pp) = \emptyset$. 
\end{itemize}

\end{remark}

\begin{remark} One can use our arguments to prove the analogue of Clifford's theorem in its usual form: let $D_\pp\in\Div(\scrX_{\pp}^{\text{sm}}(\F_\pp))$;
if $r_{\Ab}(K_{\scrX_\pp}-D_\pp)\geq 0$, then 
$r_{\Ab}(D_\pp)\leq \frac{\deg(D_\pp)}{2}.$  The hypothesis produces a twist $\varphi$ such that 
$(K_{\scrX_\pp}-D_\pp)(\varphi)$ has a section $s_v$ on every component $X_v$.  One applies the above argument to $(K_{\scrX_\pp})|_{X_v}-(s_v)$ on each component.  
Unfortunately, the zero-locus of $s_v$ is not necessarily supported on smooth points of $\scrX_\pp$ so the above proposition does not directly apply. 
\end{remark}

\begin{corollary} 
\label{c:rankOfK}
$r_{\Ab}(K_{\scrX_\pp})=g-1$
\end{corollary}

\begin{proof}
By the specialization lemma, $r_{\Ab}(K_{\scrX_\pp})\geq g-1$.  The opposite inequality is given by Proposition \ref{p:clifford}
\end{proof}

\remark{It is a result of Amini-Baker \cite{AB_metrized} that $r_{\Ab}$ satisfies the appropriate version of the Riemann-Roch theorem.  They note that the Clifford bounds follow from this Riemann-Roch theorem.  We include an independent proof for the sake of completeness.  Additionally, our proof contains an algorithm that may give bounds sharper than Theorem \ref{T:mainthm} in specific examples.  
%
}

\remark{We do not know if $r_{\toric}$ satisfies the Riemann-Roch theorem.  In practice, one may obtain better bounds on $r_X$ by considering $r_{\toric}$.  See Subsection \ref{S:remarks} for examples.}

\section{Bounding the degree of the Chabauty Divisor}
\label{S:finalTouch}

Here we show that the degree of the Chabauty divisor $D_{\chab,\pp}$ of Subsection
\ref{ss:chabautyDivisor} is at most $2r$, completing the proof of
Theorem \ref{T:mainthm}. The rough idea of the proof is the following
two steps -- (i) passing to
a semistable model so that our results about geometric rank functions
of Section \ref{S:geometricRank} apply, and (ii) using the 
subspace $V_{\chab}$ (see Proposition \ref{P:differentials}) to show that
$g-r-1 \leq  r_{\Ab}(K_{\scrX_\pp}-D_{\chab,\pp})$, which, combined with the Clifford bound
$r_{\Ab}(K_{\scrX_\pp}-D_{\chab,\pp}) \leq
\frac{1}{2}\deg(K_{\scrX_\pp}-D_{\chab,\pp})$ and a rearrangement,
directly proves the theorem. (Actually, step (ii) is a bit more subtle
that this, and we must make use of the `avoidance' variant of
Clifford's theorem of Remark \ref{r:clifford}.)\\

 In the following we continue with the setup of  Section \ref{S:method} to prove the theorem in the semistable case.

\begin{proposition} 
\label{P:mainSemistable}
Let $\scrX$ be a regular and semistable curve over $\calO_{K_{\pp}}$.   Let $D_{\chab,\pp}$ be the Chabauty divisor (see Equation \ref{eq:chabautyDivisor}).  Then $\deg D_{\chab,\pp}\leq 2r$.
\end{proposition}

\begin{proof}
By using the avoidance form of the Clifford bounds as in Remark \ref{r:clifford}, we obtain an effective divisor $E_\pp \in \Div(\scrX^{\sm}_\pp(\F_\pp))$  such that
\begin{itemize}
\item [(i)]   $\deg E_\pp \leq \frac{1}{2}\deg(K_{\scrX_\pp}-D_{\chab,\pp})$,  
\item [(ii)] For any twist $\varphi$, $\calO_{\scrX}(\varphi)|_{X_v}\otimes \calO_{\scrX_\pp}(K_{\scrX_\pp}-D_{\chab,\pp} - E_\pp)|_{X_v}$ has no non-zero sections for some component $X_v$ of $\scrX_\pp$, and 
\item [(iii)] $\Supp(E_\pp) \cap \Supp(K_{\scrX_\pp}-D_{\chab,\pp}) = \emptyset$. 
\end{itemize}
Moreover, we claim that $\deg(E_\pp) \geq \dim V_{\chab}-1$.  Indeed,
suppose not, and let $E \in \Div(X_\pp(K_\pp))$ be an effective
divisor such that $\rho(E) = E_\pp$. Then there would exist an element
$\omega$ of $V_{\chab}$ vanishing on $E$ (with the correct
multiplicities).   We may view $\omega$ as a rational section of $K_{\scrX}$.  By picking a twist $\varphi$ 
as in the proof of Proposition \ref{L:upperBounds}, we can treat $\omega$ as a regular section of $K_{\scrX}(\varphi)$
vanishing along no component of the special fiber.  For any component $X_v$, $\omega$ vanishes on $E_\pp$ and on $D_{\chab,\pp}$ by construction.  Since $\Supp\rho(E_\pp)$ is disjoint from $\Supp D_{\chab,\pp}$, $\omega$ restricts to
a non-zero section of $\calO_{\scrX}(\varphi)|_{X_v}\otimes \calO_{\scrX_\pp}(K_{\scrX_\pp}-D_{\chab,\pp} - E_\pp)|_{X_v}$ for each $v$.  This contradicts condition
(ii) in the choice of $E_\pp$.

%

Putting these inequalities together gives
\[
\deg(K_{\scrX_\pp}-D_{\chab,\pp}) \geq 2\deg(E_\pp) \geq 2\dim V_{\chab}-2 \geq 2g-2r-2.
\]
Rearranging completes the proof.
\end{proof}

Now we remove the semistability hypothesis.

\begin{theorem} Let $\scrX$ be a proper regular curve over over $\calO_{K_{\pp}}$ with Chabauty divisor $D_{\chab,\pp}$. Then $\deg D_{\chab,\pp} \leq 2r$.
\end{theorem}

\begin{proof}

By replacing $K_\pp$ by its maximal unramified extension, we may suppose that the residue field of $\OO_{K_\pp}$ is algebraically closed.
Let $K'_\pp$ be a finite extension of $K_\pp$ over which the base extension $X_\pp\times_{K_\pp} K_\pp'$ has regular semistable reduction. Denote the valuation on $K'_\pp$ by $v'$.  Let $\scrX'$ be the base
change $\scrX \times_{\calO_{K_\pp}}\calO_{K'_\pp}$ of $\scrX$ to
$\calO_{K'_\pp}$. (Note that $\calO_{K'_\pp}$ is a DVR with infinite
residue field $\F_{\pp}'$.) Then $\scrX'$ may no longer be
regular. However, there exists a dominant proper map 
$\scrX'' \to \scrX'$ of curves over $\calO_{K_\pp}'$ such that $\scrX''$ is regular and \emph{semistable}
(though probably not minimal) and  which is an isomorphism on generic
fibers. Define $V_{\chab}' = K_\pp'\otimes_{K_\pp} V_{\chab}$ and
define $D_{\chab,\pp}'$ as in Definition \ref{D:chabautyDivisor}; i.e., 
\[
D'_{\chab,\pp} = \sum_{\widetilde{Q} \in \scrX^{''\sm}(\F_{\pp}')} 
\left(\min\left\{n(\omega',\widetilde{Q}) 
: \omega' \in V'_{\chab}\right\}\right) \widetilde{Q}.
\]


We claim that $\deg D_{\chab,\pp} \leq \deg D'_{\chab,\pp}$ from which
the theorem would follow from Proposition \ref{P:mainSemistable}.  In
fact, the the pullback of $D_{\chab,\pp}$ to $\scrX''_{\pp}$ is a
subdivisor of $D'_{\chab,\pp}$.   Indeed, let $\widetilde{Q}$ be a
point of $\scrX^{''\sm}(\F_{\pp}')$ 
and denote by $X_w$ the component of $\scrX^{''}_{\pp}$
containing $\widetilde{Q}$.
Now, let
$\omega' \in V_{\chab}'$. Then $\omega'$ is a finite sum $\sum
a_i\otimes\omega_i$, for $\omega_i \in V_{\chab}$ and $a_i \in
K'_\pp$.  By associating elements of $K_\pp$ across the tensor
product, we may suppose that each $\omega_i$ extends to $\scrX$ and
that none vanish along $X_w$. 
Moreover, choose an expression $\sum a_i\otimes\omega_i$ for
$\omega'$ such that $\min_i v'(a_i)$ is maximized.  Now express each
$a_i$ as $(\pi')^{k_i}u'_i$ where $\pi'$ is a uniformizer of $K'_\pp$
and $u'_i$ has valuation equal to $0$.  Let $k$ be the minimum of the
$k_i$'s.  

We claim that  $\sum_{i|k_i=k} u'_i\otimes\omega_i$ does not vanish
along $X_w$.  If it did, for each $i$ with $k_i=k$, take
$u_i\in K_{\pp}^{\ur}$ that has the same residue as $u'_i$ and rewrite the sum of terms with minimum $v'(a_i)$ as 
\[(\pi')^k\left(\sum_{i|k_i=k} (u'_i-u_i)\otimes\omega_i+\sum_{i|k_i=k} u_i\otimes\omega_i\right).\]  
Since $\sum u_i\otimes\omega_i=1\otimes \sum u_i\omega_i$ vanishes along $X_w$, it can be
written as $(\pi')^l\otimes\tilde{\omega}$ for some $l>0$ and $\tilde{\omega}$ in
$V_{\chab}$ where $\tilde{\omega}$ does not vanish along
$X_w$.  Moreover, each $u'_i-u_i$ has positive valuation.
Consequently, $\omega$ can be written as a linear combination with
higher $\min(v'(a_i))$ contradicting our choice of expression for
$\omega$. 

 Now, since $\sum_{i|k_i=k} u'_i\omega_i$ does not vanish along $X_w$, by the non-Archimedean property of vanishing orders, 
$v_{\widetilde{Q}}(\omega')  \geq \min_{i|k_i=k}{v_{\widetilde{Q}}(\omega_i)}$.
Thus  the claim follows. 

\end{proof}

\begin{remark} One can obtain this bound using linear systems on weighted graphs as developed by Amini and Caporaso \cite{AC_weighted}.  Their theory is more combinatorial than ours and does not consider the algebraic geometry of the irreducible components.  Our approach, however, may give sharper bounds in examples when one has an understanding of the geometry of the components.  \end{remark}

\begin{remark} We describe in this remark a rank function that is
  constructed to give the sharpest bound in situations like those
  considered above.  Note that the divisor $D_{\chab,\pp}$ is supported on the special fiber.  Given a line bundle $\calL$ on a semistable model $\scrX$ and an effective divisor $D_\pp$ supported on smooth points of the special fiber, one can define a rank function $r(\calL,D_\pp)$ by saying that $r(\calL,D_\pp)\geq r$ if and only if for every effective divisor $E_\pp\in\Div(\scrX_\pp^{\text{sm}}(\F_\pp))$ of degree $r$, there is a section $s$ of $\calL|_X$ such that $D_\pp+E_\pp$ is contained (with multiplicity) in the reduction of the divsior $(s)$.   Indeed, the proof of Proposition \ref{P:mainSemistable} uses a bound on $r(K_{\scrX_\pp},D_{\chab,\pp})$.
\end{remark}

\section{Examples}
\label{S:examples}

    \begin{example}
\label{ex:sharp}
	Here we give an example of a hyperelliptic curve with bad reduction where the refined bound
	of Theorem \ref{T:mainthm}
	is sharp. Let $X$ be the smooth genus 3 hyperelliptic curve over $\Q$ with affine piece 
	\[
	   -2 \cdot 11 \cdot 19 \cdot 173 \cdot y^2 = (x-50)(x-9)(x-3)(x+13)(x^3 + 2x^2 + 3x + 4).
	\]
	This curve has bad cuspidal reduction at the prime 5 and its regular proper minimal model 
	$\scrX$ over $\Z_5$ is given by the same equation as the above Weierstrass model. 
	A descent calculation using Magma's \verb+TwoSelmerGroup+ function 
	shows that its Jacobian has rank 1. A point count reveals that 
\[
X(\Q) \supset \{ \infty, (50,0), (9,0),(3,0),(-13,0),
(25 , 20247920 ), (25 , -20247920 )
\}
\]
(and in particular $7 \leq \#X(\Q)$)  
	and $\#\scrX^{\text{sm}}_5(\F_5) = 5$. Theorem \ref{T:mainthm} reads 
	\[
	    7 \leq \#X(\Q) \leq \#\scrX_5^{\text{sm}}(\F_5) + 2 = 7,
	\]
	which determines $X(\Q)$.
\vspace{3pt}
	
	Let $J$ be the Jacobian of $X$. Then $J$ is absolutely simple. Indeed, $J$ has good 
	reduction at 13 and for $i \in \{1, \ldots, 30\}$ a computation reveals that the 
	characteristic polynomial of Frobenius for $J_{\F_{13^i}}$ is irreducible. By an argument analogous 	to \cite{PS_cyclic}*{Proposition 14.4} (see also \cite{stoll_cycles}*{Lemma 3}) we conclude 
	that $J_{\F_{13}}$ (and hence $J$) is absolutely simple. 
	
	One can check that 5 is the only prime at which our refinement of the Chabauty-Coleman bound is sharp. 
	Thus, one can use neither a map to a curve of smaller genus nor the Chabauty-Coleman 
	bound at a prime of good reduction to determine $X(\Q)$.
    \end{example}




\easterEgg{A paper should say "that I'm being nice to you, but I can seriously kick your ass if I wanted to". -- Eric Katz}

\subsection{Remarks on obtaining better bounds}
\label{S:remarks}


    
  \begin{remark} 

Let $X$ be a smooth projective geometrically integral curve over an
algebraically closed field. For an integer $0 \leq i \leq g-1$, let 
\[
r(i) = \max\{\deg D  \text{ s.t. } D \geq 0 \text{ and } r(K_{X}-D) \geq g-i -1\};
\]
Clifford's theorem implies  that $r(i) \leq 2i$. In fact  $r(i) = 2i$ if and only if $X$ is hyperelliptic, with the
    maximum witnessed when $D$ is a sum of pairs
    of points which are conjugate under the hyperelliptic involution; see \cite{hart}*{Chapter III, Theorem 5.4}.
The true value of $r(i)$ depends on the geometry of $X$; see \cite{stoll_indep}*{Section 3}. 


  \end{remark}

The situation is even more delicate (but usually favorable) when $\scrX_\pp$ has multiple components;
 next, we
give examples where the inequalities in $r_{\num}(D)\geq r_{\Ab}(D)\geq r_{\toric}(D)\geq r_X(D)$ are strict.

\begin{example}[$r_{\num}(D)>r_{\Ab}(D)$] 
\label{E:numAb}
In general, on a semistable curve $\scrX$ which is not totally degenerate
  (i.e $g(X_v) > 0$ for some component $X_v$ of $\scrX_{\pp}$)
  one expects to find divisors $D$ such that $r_{\num}(D)>r_{\Ab}(D)$.
For example, let $\scrX$ be a genus $1$ curve with good reduction, so that $\Gamma$ is a single vertex.  Pick $P\in X_\pp(K_\pp)$.  Then $r_{\Ab}(P)=0$ since there is no non-constant section of $\OO_{\scrX_\pp}(P)$ while $r_{\num}(P)=1$.
\end{example}

\begin{example}[$r_{\Ab}(D)>r_{\toric}(D)$] Let $\scrX$ be a regular
  curve which is a regular proper minimal model of its generic fiber
  and  whose special fiber has the following form: one component
  $X_v$ is a hyperelliptic curve; the other component $X_w$ is a rational
  curve; the two components meet in two nodes which are neither fixed
  points of the hyperelliptic involution on $X_v$ nor
  hyperelliptically conjugate on $X_v$.  Note that the special fiber
  is not hyperelliptic since the hyperelliptic involution does not
  extend from $X_v$ (which would follow from the universal property of
  the regular proper minimal model).  Let $P,Q\in \scrX_\pp^{\sm}(\F_{\pp})$ be hyperelliptically conjugate points on $X_v$.  Since there is a section of $\OO_{X_v}(P+Q)$, $r_{\Ab}(P+Q)=0$.  We claim that $r_{\toric}(P+Q)=-1$.  Notice that no non-trivial twist of $P+Q$ has non-negative degree on both components.  Therefore, we only have to consider sections of $\OO_{\scrX_\pp}(P+Q)$.
A section of $\OO_{\scrX_\pp}(P+Q)|_{X_v}$ gives a hyperelliptic projection which takes different values on the nodes.  A section of $\OO_{\scrX_\pp}(P+Q)|_{X_w}=\OO_{X_w}$ is just a constant function.  Therefore, no choice of sections match across both nodes.
\end{example}

\begin{example}[$r_{\toric}(D)>r_X(D)$]Let $\scrX$ be a non-hyperelliptic curve with good reduction whose special fiber is hyperelliptic.  Let $P\in\scrX(K_{\pp})$ be a point whose reduction $\rho(P)$ to the special fiber is fixed by the hyperelliptic involution.  Then $r_{\toric}(2P)=1$ since for any $Q\in \scrX_\pp(\F_\pp)$, $\OO_{\scrX_{\pp}}(\rho(2P)-Q)$ has a non-vanishing section.  However, any section of 
$\OO_{\scrX}(2P)$ must restrict to the special fiber as a constant section.  If such a section vanishes on $Q$, then it must vanish identically on $\scrX$.  \end{example}



\section*{Acknowledgments}
We would like to thank Bjorn Poonen and Anton Geraschenko for many useful conversations, Xander Faber and  Michael Stoll for helpful comments and corrections on earlier drafts, and
Dino Lorenzini for suggesting some references for the non-semistable case and encouragement and comments.  Particular thanks should be given to Matt Baker for introducing the first author to Chabauty's method and for suggesting to the second author that the totally degenerate case might follow from \cite{BakerN:RR}.
 Computations for Example \ref{S:examples} were done using 
	the software Magma \cite{Magma}.



\end{document}